\documentclass[a4paper]{article}
\usepackage{amsthm,amssymb,amsmath,enumerate}
\usepackage{algorithm}

\usepackage{cite}

\usepackage{tikz}
\usetikzlibrary{calc}
\usetikzlibrary{backgrounds}

\colorlet{hellgrau}{black!20!white}
\colorlet{dunkelgrau}{black!60!white}

\colorlet{grau}{black!40!white}
\colorlet{bold}{black} 
\tikzstyle{ledge}=[thick, grau]
\tikzstyle{rededge}=[very thick, bold]
\tikzstyle{lvertex}=[thick,circle,inner sep=0.cm, minimum size=2mm, fill=white, draw=grau]
\tikzstyle{redvx}=[thick,circle,inner sep=0.cm, minimum size=2mm, fill=white, draw=bold]

\tikzstyle{hvertex}=[thick,circle,inner sep=0.cm, minimum size=2mm, fill=white, draw=black]
\tikzstyle{hedge}=[very thick]
\tikzstyle{medge}=[thick]
\tikzstyle{harrow}=[thick,arrows=->]
\tikzstyle{darrow}=[thick,arrows=<-]
\tikzstyle{point}=[draw,circle,inner sep=0.cm, minimum size=1mm, fill=black]
\tikzstyle{pointer}=[thick,->,shorten >=2pt,color=dunkelgrau]
\tikzstyle{facebdry}=[color=auchblau, very thick] 
\tikzstyle{face}=[facebdry,fill=hellblau]
\tikzstyle{nface}=[color=hellblau,fill=hellblau,thick] 
\tikzset{>={latex}}
\tikzstyle{tinyvx}=[thick,circle,inner sep=0.cm, minimum size=1.3mm, fill=white, draw=black]
\tikzstyle{smallvx}=[hvertex,minimum size=1.7mm]

\pgfdeclarelayer{background}
\pgfdeclarelayer{foreground}
\pgfsetlayers{background, main, foreground}

\newcommand{\comment}[1]{}
\newcommand{\N}{\mathbb N}

\newcommand{\sm}{\setminus}

\title{Connectivity of graphs that do not have the edge-Erd\H{o}s-P\'{o}sa property}
\author{Henning Bruhn and Raphael Steck}
\date{\today}

\usepackage{hyperref}
\hypersetup{
pdftitle={Connectivity of graphs that do not have the edge-Erdös-P\'{o}sa property}, 
pdfauthor={Raphael Steck},
pdfsubject={Connectivity and edge-Erdös-P\'{o}sa property},
pdfproducer={pdfeTex 3.14159-1.30.6-2.2},
colorlinks=false,
pdfborder=0 0 0	
}

\usepackage{enumitem}

\theoremstyle{plain}
\newtheorem{theo}{Theorem}

\newtheorem{rem}[theo]{Remark}

\begin{document}

\maketitle

\begin{abstract}
We show that we can assume graphs that do not have the edge-Erd\H{o}s-P\'{o}sa property to be connected. Then we strengthen this result to $2$-connectivity under the additional assumptions of a minor-closed property and a generic counterexample.
\end{abstract}




A class $\mathcal{F}$ has the \emph{edge-Erd\H{o}s-P\'{o}sa property} if there exists a function $f: \N \rightarrow \mathbb{R}$ such that for every graph $G$ and every integer $k$, there are $k$ edge-disjoint graphs in $G$ each isomorphic to some graph in $\mathcal{F}$ or there is an edge set $X \subseteq E(G)$ of size at most $f(k)$ meeting all subgraphs in $G$ isomorphic to some graph in $\mathcal{F}$. The edge set $X$ is called the \emph{hitting set}.
If we replace vertices with edges in the above definition, that is, we look for a vertex hitting set or vertex-disjoint graphs, then we obtain the \emph{vertex-Erd\H{o}s-P\'{o}sa property}.
The class $\mathcal{F}$ that is studied in this paper arises from taking minors: For a fixed graph $H$, we define the set 
\[
\mathcal{F}_H = \{ G \, | \, H \text{ is a minor of } G\}.
\]

In other words, $\mathcal{F}_H$ is the set of \emph{$H$-expansions}. The vertex-Erd\H{o}s-P\'{o}sa property for $\mathcal{F}_H$ is well understood: Robertson and Seymour \cite{robertson86} proved that the class $\mathcal{F}_H$ has the vertex-Erd\H{o}s-P\'{o}sa property if and only if $H$ is planar. This implies that the vertex-Erd\H{o}s-P\'{o}sa property is closed under taking minors, which in turn implies that 

\begin{equation}\label{vertexEPPconnectiviy}
\begin{minipage}[c]{0.8\textwidth}\em
If for a graph $H$ the class $\mathcal{F}_H$ has the vertex-Erd\H{o}s-P\'{o}sa property, then so does the class $\mathcal{F}_C$ for every component $C$ of $H$.
\end{minipage}\ignorespacesafterend 
\end{equation}

For the edge-Erd\H{o}s-P\'{o}sa property, it is not known whether it is minor-closed or not, and it is not at all clear whether it should be. Thus, we tackle \eqref{vertexEPPconnectiviy} for the edge-Erd\H{o}s-P\'{o}sa property. We show that

\begin{theo} \label{con:theo:1connected}
If for a graph $H$ the class $\mathcal{F}_H$ has the edge-Erd\H{o}s-P\'{o}sa property, then so does $\mathcal{F}_C$ for every component $C$ of $H$.
\end{theo}


Interestingly, Robertson and Seymour proved their result about the vertex-Erd\H{o}s-P\'{o}sa property of planar graphs in two steps: First, they proved it for every connected planar graph. In a second step, they lifted the connectivity requirement. Thus Theorem~\ref{con:theo:1connected} might also provide some help in verifying for which graphs $H$ the class $\mathcal{F}_H$ has the edge-Erd\H{o}s-P\'{o}sa property. For example, it might help to prove that $\mathcal{F}_H$ does not have the edge-Erd\H{o}s-P\'{o}sa property if $H$ has large treewidth (for larger or arbitrary maximum degree of $H$).

To this end, we also attempt a strengthening of Theorem~\ref{con:theo:1connected} which allows us to not only focus on connected, but $2$-connected graphs $H$. This, however, is not achieved in full generality, and we prove it only by imposing some additional assumptions (see Theorem~\ref{con:theo:2connected}).

For a graph $H$, a connected graph $G \subseteq H$ with at least one vertex $v \in V(G)$ with $d_H(v) \geq 3$ and an integer $r \in \N$, we define $G^{\times}$ to be the following graph: Starting with the empty graph, we add a copy of every vertex $v \in V(G)$ with $d_H(v) \geq 3$ to $G^{\times}$. For every non-trivial path $P$ of length $l$ in $G$ between two vertices such vertices, we add $r$ internally disjoint paths of length $\max\{l, 2\}$ between their corresponding copies in $G^\times$. Finally, for every path $P$ of length $l$ in $G$ between a vertex $u$ with $d_G(u) = 1$, $d_H(u) \leq 2$ and its closest vertex $v \in V(G)$ with $d_H(v) \geq 3$, we add $r$ paths of length $l$ which are disjoint except for the copy of $v$. 
See Figure~\ref{fig:Gtimesconstruction} for an example.

\begin{figure}[hbt] 
\centering
\begin{tikzpicture}[scale=0.85]
\tikzstyle{tinyvx}=[thick,circle,inner sep=0.cm, minimum size=1.3mm, fill=white, draw=black]

\node[hvertex,fill=hellgrau,label=below:$v$] (v) at (-4,0) {};
\node[hvertex,fill=hellgrau,label=below:$w$] (w) at (-2,0) {};
\node[tinyvx,label=left:$u$] (v1) at (-5,-1) {};
\node[hvertex,fill=hellgrau] (w1) at (-1,-1) {};
\node[hvertex,fill=hellgrau] (u1) at (-4,2) {};
\node[hvertex,fill=hellgrau] (u2) at (-2,2) {};

\draw[hedge] (v) -- (w) -- (w1);
\draw[hedge] (v) -- (u1) -- (u2) -- (w);
\draw[dashed] (v1) -- (v);

\node[hvertex,fill=hellgrau,label=below:$v^*$] (v*) at (2,0) {};
\node[hvertex,fill=hellgrau,label=below:$w^*$] (w*) at (4,0) {};

\node at (-6, 1) {$G$};
\node at ( 6, 1) {$G^\times$};

\tikzstyle{vvx}=[thick,circle,inner sep=0.cm, minimum size=1.5mm, fill=white, draw=black]
\def\blabb{0.3}

%

\begin{scope}[shift={(2,2)},rotate=45]

\foreach \i in {0,...,3}{
  \node[smallvx] (xy\i) at (0,\i*\blabb-1.5*\blabb) {};
  \draw[hedge] (v*) to (xy\i);
}
\end{scope}

\begin{scope}[shift={(4,2)},rotate=315]

\foreach \i in {0,...,3}{
  \node[smallvx] (yx\i) at (0,\i*\blabb-1.5*\blabb) {};
  \draw[hedge] (xy\i) to (yx\i) to (w*);
}
\end{scope}

\begin{scope}[shift={(5,-1)},rotate=135]

\foreach \i in {0,...,3}{
  \node[smallvx] (yy\i) at (0,\i*\blabb-1.5*\blabb) {};
  \draw[hedge] (yy\i) to (w*);
}
\end{scope}

\begin{scope}[shift={(3,0)},rotate=0]

\foreach \i in {0,...,3}{
  \node[smallvx] (zz\i) at (0,\i*\blabb-1.5*\blabb) {};
  \draw[hedge] (v*) to (zz\i) to (w*);
}
\end{scope}

\end{tikzpicture}
\caption{Construction of $G^\times$, with $u \in V(H) \sm V(G)$.}
\label{fig:Gtimesconstruction}
\end{figure}
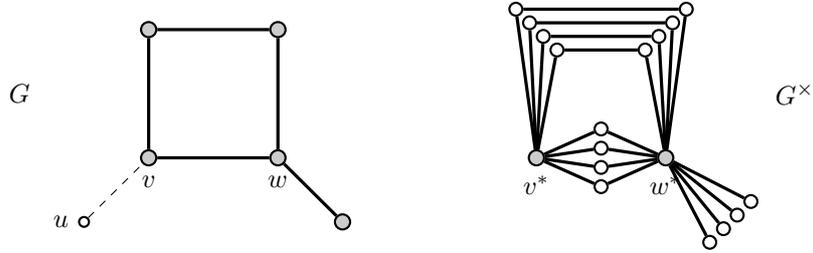

Let us check that for every edge set $X$ of size at most $r-1$, $G^\times - X$ contains a $G$-expansion. Every $v \in V(G)$ with $d_H(v) \geq 3$ can be mapped to its copy $v' \in V(G^\times)$. Every $u$--$v$~path between two such vertices can be mapped to one of its copies in $G^{\times}$ that is disjoint from $X$. For every vertex $u \in V(G)$ with $d_G(u) = 1$ and $d_H(u) \leq 2$, there is vertex $v \in V(G)$ with $d_H(v) \geq 3$ that is closest to $u$. Among all copies of the $u$--$v$~path, we pick a copy $P'$ that is disjoint from $X$ and map $P$ to $P'$ such that $v$ is mapped to $v'$. The prerequisite of $G$ being connected and containing a vertex $v$ with $d_H(v) \geq 3$ implies that every $v \in V(G)$ with $d_G(v) = d_H(v) = 2$ lies on a path in $G$ between vertices of degree other than $2$ in $H$, with at least one endvertex of the path having degree at least $3$ in $H$. We conclude that the above mapping yields a $G$-expansion in $G^\times - X$.

For $r \geq 3$, the number of vertices of degree at least~$3$ in $G^\times$ is the same as the number of vertices $v$ with $d_H(v) \geq 3$ in $G$. 

\begin{rem} \label{con:rem:degree3verticesExist}
Let $H$ be a connected graph for which $\mathcal{F}_H$ does not have the edge-Erd\H{o}s-P\'{o}sa property. Then $H$ contains vertices of degree at least~$3$.
\end{rem}

\begin{proof}
Suppose $H$ would only contain vertices of degree~$2$ or less. Since $H$ is connected, $H$ must be a cycle, a path or an isolated vertex. However, for all of those graphs, $\mathcal{F}_H$ is already known to have the edge-Erd\H{o}s-P\'{o}sa property.
%
%
\end{proof}

For two graphs $A$ and $B$, we define $\preceq$ by
\[
A \preceq B \Leftrightarrow A \text{ is a minor of } B.
\]
Similarly, we define $\not\preceq$ by
\[
A \not\preceq B \Leftrightarrow A \text{ is not a minor of } B.
\]

\section{1-Connectivity}

\begin{proof}[Proof of Theorem~\ref{con:theo:1connected}]
Let $A$ be some component of $H$ such that $\mathcal{F}_A$ does not have the edge-Erd\H{o}s-P\'{o}sa property. Thus, there exists an integer $k \in \N$ such that for every $r \in \N$, there exists a graph $A^*_r$ such that $A^*_r$ neither contains $k$ edge-disjoint expansions of $A$ nor an edge set $X$ of size at most~$r-1$ such that $A^*_r-X$ contains no expansion of $A$.
We separate the other components of $H$ into two disjoint sets $\mathcal{B}$ and $\mathcal{C}$, which we define by
\begin{align*}
\mathcal{B} &= \{B \text{ component of }H \,|\, A \not\preceq B\} \text{ and} \\
\mathcal{C} &= \{C \text{ component of }H \,|\, A \preceq C\} \setminus \{A\}.
\end{align*}
To prove that $\mathcal{F}_H$ does not have the edge-Erd\H{o}s-P\'{o}sa property, let $r$ be given. We prove that there exists a graph $H^*$ that contains neither $k$ edge-disjoint $H$-expansions nor an edge set $X$ of size at most~$r-1$ such that $H^*-X$ contains no $H$-expansion.
We define $H^*$ to be the disjoint union of
\begin{itemize}
\item $A^* = A^*_r$,
\item for every $B \in \mathcal{B}$: $r$ distinct copies of $B$ and
\item for every $C \in \mathcal{C}$: one $C^\times$. 
\end{itemize}

First we show that $H^*$ does not contain an edge-hitting set meeting all $A$-expansions. Let $X \subseteq E(H^*)$ be an edge set of size at most~$r-1$. We claim that $H^* - X$ still contains an $H$-expansion: Indeed, there is an $A$-expansion in $A^* - X$ by choice of $A^*$. For every $B \in \mathcal{B}$, at least one of the $r$ copies of $B$ in $H$ is disjoint from $X$. Thus, there is a $B$-expansion in that copy. Finally, for every $C \in \mathcal{C}$, there is a $C$-expansion in $C^\times - X$ by construction of $C^\times$. Together, this yields an $H$-expansion in $H^*-X$.

We claim
\begin{equation}\label{AinA*for1connected}
\begin{minipage}[c]{0.8\textwidth}\em
\begin{center}
Every $H$-expansion in $H^*$ contains an $A$-expansion in $A^*$.
\end{center}
\end{minipage}\ignorespacesafterend 
\end{equation} 

Note that \eqref{AinA*for1connected} finishes the proof of the theorem: Indeed, by choice of $A^*$, there can be no $k$ edge-disjoint $A$-expansions in $A^*$.
To prove the claim, consider an expansion $H'$ in $H^*$, and suppose that \eqref{AinA*for1connected} is false for $H'$.

Since every component of $H'$ is connected, it must be contained in a single component of $H^*$. Further note that every expansion of a $C \in \mathcal{C}$ in some component of $H^*$ contains an expansion of $A$ by definition of $\mathcal{C}$. Thus, by definition of $\mathcal{B}$, no $A$-expansion (and thus no $C$-expansion for any $C \in \mathcal{C}$) can be contained in some copy of some $B \in \mathcal{B}$. On top of that, if an $A$-expansion (or a $C$-expansion for any $C \in \mathcal{C}$) is embedded in $A^*$, this proves the above claim. Thus, suppose all $A$-expansion in $H'$ (and thus all $C$-expansions for every $C \in \mathcal{C}$) are contained in $\bigcup\limits_{C \in \mathcal{C}} C^\times$. Every $H$-expansion in $H^*$ contains at least one vertex of degree~$\geq 3$ in $H^*$ for every vertex of degree~$\geq 3$ in $H$. However, by construction of $C^\times$, $\bigcup\limits_{C \in \mathcal{C}} C^\times$ contains no more vertices of degree~$\geq 3$ than $\bigcup\limits_{C \in \mathcal{C}} C$. Since $A$ contains vertices of degree $\geq 3$ by Remark~\ref{con:rem:degree3verticesExist} and the branch sets of all vertices in $A \cup \left(\bigcup\limits_{C \in \mathcal{C}} C\right)$ must be contained in $\bigcup\limits_{C \in \mathcal{C}} C^\times$, this is a contradiction. Thus \eqref{AinA*for1connected} holds, proving the theorem.
\end{proof}


\section{2-Connectivity}

If we want to prove that for some graph $H$, $\mathcal{F}_H$ does not have the edge-Erd\H{o}s-P\'{o}sa property, the above theorem implies that it suffices to check the components of $H$ individually. Thus, without loss of generality, we can assume that $H$ is $1$-connected.
To improve this to $2$-connectivity of $H$, we need two additional assumptions.

First, we define a graph property $\mathcal{P}$ to be \emph{hereditary} if $\bar{\mathcal{P}}$ is closed under taking minors, that is, for every graph $H$ without property $\mathcal{P}$, no minor of $H$ has property $\mathcal{P}$. An example for a hereditary property would be large treewidth: If a graph $H$ does not have treewidth at least $t$, then no minor of $H$ has treewidth at least $t$.

Second, for a block $A$ of a graph $H$, let $S$ be the set of cutvertices of $H$ that lie in $A$. We say that there is a \emph{generic counterexample} for $A$ if there exists some $k \in \N$ such that for every $r \in \N$, there is a graph $A^*$ with the following properties: There are no $k$ edge-disjoint $A$-expansions in $A^*$. Furthermore, for every $s \in S$, there is a $s' \in V(A^*)$ such that for every edge set $X$ of size at most $r-1$, there is an embedding of $A$ in $A^*$ such that for every vertex $s \in S$, the branch set $B_s$ contains $s'$. (Together, this implies that $\mathcal{F}_A$ does not have the edge-Erd\H{o}s-P\'{o}sa property.)
Every known construction that shows that some class $\mathcal{F}_A$ does not have the edge-Erd\H{o}s-P\'{o}sa property is a generic counterexample.

\begin{theo} \label{con:theo:2connected}
Let $\mathcal{P}$ be a hereditary graph property and let $H$ be a graph that contains a block with property $\mathcal{P}$.
Furthermore, for every block $A$ of $H$ with property $\mathcal{P}$, let there be a generic counterexample.

Then $\mathcal{F}_H$ does not have the edge-Erd\H{o}s-P\'{o}sa property.
\end{theo}

An example for an application of Theorem*\ref{con:theo:2connected} could be the following. Suppose we want to show that:
\begin{equation}\label{con:application:treewidth}
\begin{minipage}[c]{0.8\textwidth}\em
For every graph $H$ of treewidth at least $10^{100}$, $\mathcal{F}_H$ does not have the edge-Erd\H{o}s-P\'{o}sa property.
\end{minipage}\ignorespacesafterend 
\end{equation} 

Having large treewidth is a hereditary graph property. Assume we were able to show that for every $2$-connected graph $H$ of treewidth at least $10^{100}$, $\mathcal{F}_H$ does not have the edge-Erd\H{o}s-P\'{o}sa property. Then we will most likely do that by giving a generic counterexample. Thus we can apply Theorem~\ref{con:theo:2connected} to drop the connectivity requirement and we obtain~\eqref{con:application:treewidth}. Now let us prove Theorem~\ref{con:theo:2connected}.

\begin{proof}
We can assume $H$ to be $1$-connected: Indeed, if $H$ contains a block $A$ with property $\mathcal{P}$, then there is a component $Q$ of $H$ that contains $A$. Furthermore, if all block of $H$ with property $\mathcal{P}$ allow for a generic counterexample, then this includes the blocks of $Q$. If we are able to prove that $\mathcal{F}_Q$ does not have the edge-Erd\H{o}s-P\'{o}sa property, then using Theorem~\ref{con:theo:1connected}, we conclude that $\mathcal{F}_H$ does not have the edge-Erd\H{o}s-P\'{o}sa property.

We consider the block tree $T$ of $H$. Let $T_\mathcal{P}$ be the minimal subtree of $T$ that contains all blocks with property $\mathcal{P}$. We pick $A$ to be a leaf of $T_\mathcal{P}$. Note that $A$ is a block of $H$ that has the property $\mathcal{P}$. Observe that $A$ is not a trivial block, i.e. $A$ is not a $K_2$ due to Remark~\ref{con:rem:degree3verticesExist}.

We define
\begin{align*}
\mathcal{C} &= \{C \text{ block of }H \,|\, A \preceq C\} \setminus \{A\}.
\end{align*}

Since $\mathcal{P}$ is hereditary, it holds for every $C \in \mathcal{C}$ that $C$ has the property $\mathcal{P}$, too. 
Let $T_\mathcal{C}$ be the minimal subtree of $T$ that contains all blocks of $\mathcal{C} \cup \{A\}$. We observe that $T_\mathcal{C}$ is a subgraph of $T_\mathcal{P}$. Thus, 
\begin{equation}\label{AisLeafOfT_C}
\begin{minipage}[c]{0.8\textwidth}\em
\begin{center}
$A$ is a leaf of $T_\mathcal{C}$.
\end{center}
\end{minipage}\ignorespacesafterend 
\end{equation} 

We define
\begin{align*}
\mathcal{B} = &\{B \text{ block of }H \,|\, A \not\preceq B\} \cap V(T_\mathcal{C}) \text{ and} \\
\mathcal{D} = &\{D \text{ component of } \bigcup_{\substack{B \text{ block of } H \\ B \not\in V(T_\mathcal{C})}} B\}.
\end{align*}

Note that since $T$ is a tree, $T - T_C$ is a forest. For each component $T'$ in $T - T_C$, $V(T')$ are the blocks and cutvertices of one element of $\mathcal{D}$.

To show that $H$ does not have the edge-Erd\H{o}s-P\'{o}sa property, let $r \geq 3$ be some integer. We define our counterexample graph $H^*$ to be the union of
\begin{itemize}
\item $A^*$,
\item for every $C \in \mathcal{C}$: one $C^\times$,
\item for every $D \in \mathcal{D}$: $r$ distinct copies of $D$,
\item for every non-trivial $B \in \mathcal{B}$: one $B^\times$ and
\item for every component $P$ of $\bigcup\limits_{\substack{B \in \mathcal{B} \\ B = K_2}} B$: one $P^\times$.
\end{itemize}

We pick the above graphs to be disjoint except for those vertices which are copies of the same vertex $v \in H$, which we identify with each other in $H^*$, too. In $A^*$, we pick the vertex $s'$ for every $s \in S$ and identify it with all copies of $s$. 
Note that for all blocks $B \in \mathcal{B}$, there is a path $P$ in $H$ whose endvertices are in a block in $\{A\} \cup \mathcal{C}$ and $P$ contains an edge of $B$. With $A$ and $C$ being non-trivial blocks for all $C \in \mathcal{C}$, their cutvertices have degree at least $3$ in $H$. Thus the union of all trivial blocks in $B$ is a collection of paths $P$ whose endvertices are in non-trivial blocks and have degree at least $3$ in $H$.
We denote the union of all $P^\times$ and all $B^\times$ for every non-trivial $B \in \mathcal{B}$ by $\mathcal{B}^\times$. Note that every vertex $v'$ in $\mathcal{B}^\times \cup \bigcup_{C \in \mathcal{C}} C^\times$ with $d_{H^*}(v') \geq 3$ is the copy of some vertex $v$ in $\bigcup_{B \in \mathcal{B}} B \cup \bigcup_{C \in \mathcal{C}} C$ with $d_H(v) \geq 3$.

Now we show that $H^*$ does not contain an edge-hitting set meeting all $A$-expansions. Let $X \subseteq E(H^*)$ be an edge set of size at most~$r-1$. We claim that $H^* - X$ still contains an embedding of $H$: Indeed, we can embed $A$ in $A^* - X$ such that for every $s \in S$, its branch set $B_s$ contains $s'$ by choice of $A^*$. For every $D \in \mathcal{D}$, at least one of the $r$ copies of $D$ is disjoint from $X$. Thus, we can embed $D$ in that copy. For every non-trivial $B \in \mathcal{B} \cup \mathcal{C}$, we can embed $B$ in $B^\times - X$ by construction of $B^\times$. We observed above that trivial blocks of $B$ are contained in some paths between vertices of degree at least $3$, which can be embedded in their copy in $H^* - X$. Together, this yields an embedding of $H$ in $H^* - X$ by construction of $H^*$.

It remains to show that there are no $k$ edge-disjoint embeddings of $H$ in $H^*$. For this, we claim:
\begin{equation}\label{AinA*for2connected}
\begin{minipage}[c]{0.8\textwidth}\em
\begin{center}
Every $H$-expansion in $H^*$ contains an $A$-expansion in $A^*$.
\end{center}
\end{minipage}\ignorespacesafterend 
\end{equation}

Note that \eqref{AinA*for2connected} proves the theorem: Indeed, by choice of $A^*$, there can be no $k$ edge-disjoint embeddings of $A$ in $A^*$.
Let us prove \eqref{AinA*for2connected}. Way say that a block $B$ of $H$ is \emph{embedded in a block $B^*$ of $H^*$}, when for every $v \in V(B)$ with $d_H(v) \geq 3$, the branch set $B_v$ contains a vertex $v^* \in V(B^*)$ with $d_{B^*} \geq 3$. In this sense, every block $B$ of $H$ is embedded in a single block $B^*$ of $H^*$.
Note that every embedding of a $C \in \mathcal{C}$ in some block of $H^*$ contains an embedding of $A$ by definition of $\mathcal{C}$. Thus neither $A$ nor any $C \in \mathcal{C}$ can be embedded in any copy of some $D \in \mathcal{D}$. Additionally, if some $C \in \mathcal{C}$ is embedded in $A^*$, this includes an embedding of $A$ in $A^*$, which was what we wanted.
Thus, we may assume that neither $A$ nor any $C \in \mathcal{C}$ is embedded in $A^*$.

We conclude that $A \cup \bigcup\limits_{C \in \mathcal{C}} C$ is embedded in $\mathcal{B}^\times \cup \bigcup\limits_{C \in \mathcal{C}} C^\times$. 
Let $B \in \mathcal{B}$. By definition of $\mathcal{B}$, $B$ is on the unique path in $T$ that connects two blocks $C_1, C_2 \in \mathcal{C} \cup \{A\}$. Since $A$ is a leaf of $T_\mathcal{C}$ by \eqref{AisLeafOfT_C} and we assumed that neither $C_1$ nor $C_2$ are embedded in $A^*$, $B$ cannot be embedded in $A^*$. For every $D \in \mathcal{D}$, there is a cutvertex separating $D$ from all $C \in \{A\} \cup \mathcal{C}$. Therefore, $B$ cannot be embedded in $D$.
We conclude that $B$ is embedded in $\mathcal{B}^\times \cup \bigcup\limits_{C \in \mathcal{C}} C^\times$.

To sum up, $A \cup \bigcup\limits_{B \in \mathcal{B}} B \cup \bigcup\limits_{C \in \mathcal{C}} C$ is embedded in $\mathcal{B}^\times \cup \bigcup\limits_{C \in \mathcal{C}} C^\times$. However, the number of vertices $v$ in $\bigcup\limits_{B \in \mathcal{B}} B \cup \bigcup\limits_{C \in \mathcal{C}} C$ with $d_H(v) \geq 3$ is the same as the number of vertices $v'$ in $\mathcal{B}^\times \cup \bigcup\limits_{C \in \mathcal{C}} C^\times$ with $d_{H^*}(v') \geq 3$.

By by Remark~\ref{con:rem:degree3verticesExist}, $A$ contains at least one vertex $v$ with $d_A(v) \geq 3$.
Since $A$ is $2$-connected, we conclude that it must contain at least two vertices $v$ with $d_A(v) \geq 3$.
Since $A$ is a leaf of $T_C$, it shares exactly one vertex with $\bigcup\limits_{B \in \mathcal{B}} B \cup \bigcup\limits_{C \in \mathcal{C}} C$. Thus, $A \sm \left( \bigcup\limits_{B \in \mathcal{B}} B \cup \bigcup\limits_{C \in \mathcal{C}} C \right)$ contains at least one vertex $v$ with $d_H(v) \geq d_A(v) \geq 3$.
But then it is impossible to embed $A \cup \bigcup\limits_{B \in \mathcal{B}} B \cup \bigcup\limits_{C \in \mathcal{C}} C$ in $\mathcal{B}^\times \cup \bigcup\limits_{C \in \mathcal{C}} C^\times$.

Thus, Claim~\eqref{AinA*for1connected} holds, proving the theorem.
\end{proof}

\bibliographystyle{amsplain}

\bibliography{chapters/literature}{}

%
%
%
%
%

\end{document}